\documentclass[a4paper,12 pt,twoside]{article}

\usepackage{amsmath,amssymb,amsthm}
\usepackage[T1]{fontenc}

\newtheorem{theorem}{Theorem}[section]
\newtheorem{corollary}[theorem]{Corollary}
\newtheorem{lemma}[theorem]{Lemma}
\newtheorem{proposition}[theorem]{Proposition}

\theoremstyle{definition}
\newtheorem{definition}[theorem]{Definition}
\theoremstyle{definition}
\newtheorem{remark}[theorem]{Remark}
\theoremstyle{definition}
\newtheorem{example}[theorem]{Example}

\numberwithin{equation}{section}

\title{Lyapunov functions for trichotomy with growth rates of evolution operators in Banach spaces}
\date{}
\begin{document}

%%%%%%%%%%%%%%%%%%%%%%%%%%%%%%%%%%%%%%%%%%%
%%% Here is the place for your abstract %%%
%%% ----------------------------------- %%%

\maketitle

\author{ V. Crai, M. Megan}

\begin{abstract}
The main objective of this paper is to give a characterization in terms of Lyapunov functions for trichotomy with growth rates of evolution operators in Banach spaces.
\end{abstract}
\section{Introduction}
The concept of trichotomy, firstly arose in the work \cite{saker} of R.J. Sacker and G.R. Sell in 1976. They described trichotomy for linear differential systems by linear skew-product flows. Later, S. Elaydi and O. Hajek in \cite{elaydi1} and \cite{elaydi2} gave the notion of exponential trichotomy for differential equations and for nonlinear differential systems, respectively.

As a natural generalization of exponential dichotomy (firstly introduced by O. Perron in \cite{peron}) exponential trichotomy (see \cite{ba2}, \cite{ba3}, \cite{megan-stroica}, \cite{sasu} and the references therein) is one of the most complex asymptotic properties arising from the central manifold theory (see \cite{carr} and \cite{sasu}). The general case of trichotomy with growth rates was studied by L. Barreira and C. Valls in \cite{ba2}.

In the study of trichotomy the main idea is to obtain a decomposition of the space at every moment into three closed subspaces: the stable subspace, the unstable subspace and the central manifold.

In the previous studies of uniform and nonuniform trichotomies the growth rates are always assumed to be the same type functions.

This paper considers a general concept of trichotomy (and as particular case the general concept of uniform trichotomy) with different growth rates.

The main result of the paper is a characterization for nonuniform $(h_1,h_2,h_3,h_4)$-trichotomy respectively uniform $(h_1,h_2,h_3,h_4)$-trichotomy of evolution operators in Banach spaces.
The characterizations are given in terms of Lyapunov functions. Our characterizations are inspired by the work of Ya.B. Pesin (\cite{pesin}, \cite{pesin1}) and L. Barreira, C. Valls and D. Draghicevici (\cite{ba2}, \cite{ba3}).
\section{Growth rates for evolution operators}
Let $X$ be a real or complex Banach space. The norms on $X$ and on $\mathcal{B}(X)$, the Banach algebra of all linear and bounded operators on $X$, will be denoted by $\|\cdot\|$. We also denote by $$\Delta=\{(t,s)\in\mathbb{R}_+^2,t\geq s\geq 0\}$$.

\begin{definition}
A map  $U:\Delta\to\mathcal{B}(X)$ is called \textit{an evolution operator} on $X$ if:
\begin{itemize}
	\item[$(e_1)$]$U(t,t)=I$ (the identity operator on $X$), for every $t\geq 0$
	\item[] and
	\item[$(e_2)$]$U(t,t_0)=U(t,s)U(s,t_0)$, for all $(t,s),(s,t_0)\in\Delta$.
\end{itemize}	
\end{definition}
\begin{definition}
A map $P:\mathbb{R}_+\to\mathcal{B}(X)$ is called
	\begin{itemize}
		\item[(i)] \textit{a family of projectors} on $X$ if 
		$$P^2(t)=P(t),\text{ for every } t\geq 0;$$
		\item [(ii)] \textit{invariant} for the evolution operator $U:\Delta\to\mathcal{B}(X)$ if:
		\begin{align*}
		U(t,s)P(s)x=P(t)U(t,s)x, 
		\end{align*}
		for all $(t,s,x)\in \Delta\times X$;
		\item[(iii)] \textit{strongly invariant} for the evolution operator $U:\Delta\to\mathcal{B}(X)$ if it is invariant for $U$ and for all $(t,s)\in\Delta$ the restriction of $U(t,s)$ on Range $P(s)$ is an isomorphism from Range $P(s)$ to Range $P(t)$.
	\end{itemize}
	\end{definition}
\begin{remark}
	It is obvious that if $P$ is strongly invariant for $U$ then it is also invariant for $U$. The converse is not valid (see \cite{lupa}).
\end{remark}
\begin{remark}\label{rem-proiectorstrong}
	If the family of projectors $P:\mathbb{R}_+\to\mathcal{B}(X)$ is strongly invariant for the evolution operator $U:\Delta\to\mathcal{B}(X)$ then (\cite{lupa}) there exists a map $V:\Delta\to\mathcal{B}(X)$ with the properties: 
	\begin{itemize}
		\item[$(v_1)$] $V(t,s)$ is an isomorphism from Range $ P(t)$ to Range $ P(s)$, 
		\item [$(v_2)$] $U(t,s)V(t,s)P(t)x=P(t)x$,
		\item[$(v_3)$] $V(t,s)U(t,s)P(s)x=P(s)x$,
		\item[$(v_4)$]$V(t,t_0)P(t)=V(s,t_0)V(t,s)P(t)$,
		\item[$(v_5)$]$V(t,s)P(t)=P(s)V(t,s)P(t)$,
		\item[$(v_6)$] $V(t,t)P(t)=P(t)V(t,t)P(t)=P(t)$,
		\end{itemize}
	for all $(t,s),(s,t_0)\in \Delta$ and $x\in X$.
\end{remark}
\begin{definition}
	Let $P_1,P_2,P_3:\mathbb{R}\to\mathcal{B}(X)$ be three families of projectors on $X$. We say that the family $\mathcal{P}=\{P_1,P_2,P_3\}$ is 
	\begin{itemize}
		\item [(i)] \textit{orthogonal} if
		\begin{itemize}
			\item [$(o_1)$]$P_1(t)+P_2(t)+P_3(t)=I$ for every $t\geq 0$\\
			and
			\item[$(o_2)$] $P_i(t)P_j(t)=0$ for all $t\geq 0$ and all $i,j\in\{1,2,3\}$ with $i\neq j$;
		\end{itemize}
	\item[(ii)] \textit{compatible} with the evolution operator $U:\Delta\to\mathcal{B}(X)$ if
	\begin{itemize}
		\item[$(c_1)$] $P_1$ is invariant for $U$\\
		and
		\item[$(c_2)$] $P_2,P_3$ are strongly invariant for $U$.
	\end{itemize}
	\end{itemize}
\end{definition}
If $\mathcal{P}=\{P_1,P_2,P_3\}$ is compatible with $U$ then in what follows we shall denote by
 $V_j(t,s)$ the isomorphism (given by Remark \ref{rem-proiectorstrong}) from Range $P_j(t)$ to Range $P_j(s)$ where $j\in\{2,3\}$.
\begin{definition}
We say that a nondecreasing map $h:\mathbb{R}_+\to[1,\infty)$ is a \textit{ growth rate} if 
\begin{align*}
\lim\limits_{t\to\infty}h(t)=\infty.
\end{align*}
\end{definition}
As particular cases of growth rates we remark:
\begin{itemize}
	\item [$(r_1)$] \textit{exponential rates}, i.e.
	$h(t)=e^{\alpha t}$ with $\alpha>0;$
	\item [$(r_2)$]\textit{polynomial rates}, i.e.
	$h(t)=(t+1)^\alpha$ with $\alpha>0.$
\end{itemize}
Let $\mathcal{P}=\{P_1, P_2, P_3\}$ be an orthogonal family of projectors which is compatible with the evolution operator $U:\Delta\to\mathcal{B}(X)$ and let $H=\{h_1,h_2,h_3,h_4\}$ be a set of growth rates.
\begin{definition}\label{def-crestere}
	We say that the pair $(U,\mathcal{P})$ has a  \textit{H-growth} (and we denote H-g) if there exists a nondecreasing function $M:\mathbb{R}_+\to [1,\infty)$ such that:
	\begin{itemize}
			\item [$(h_1g_1)$  ]$h_1(s)\|U(t,s)P_1(s)x\|\leq M(s)h_1(t)\|x\|$
			\item [$(h_2g_1)$  ]$h_2(s)\|V_2(t,s)P_2(t)x\|\leq M(t) h_2(t)  \|x\|$
			\item [$(h_3g_1)$  ]$h_3(s)\|U(t,s)P_3(s)x\|\leq M(s) h_3(t)\|x\|$
			\item[$(h_4g_1)$  ]$h_4(s)\|V_3(t,s)P_3(t)x\|\leq M(t) h_4(t)  \|x\|$
		\end{itemize}
	for all	$(t,s,x)\in \Delta\times X.$
\end{definition}
\begin{remark}
	As  particular cases of H-g we have:
	\begin{itemize}
		\item[(i)]\textit{ the uniform-H-growth}(and we denote by  u-H-g) when the function M is constant;
		\item[(ii)] \textit{the exponential growth} ($e.g.$) and respectively \textit{uniform exponential growth} ($u.e.g$) when $h_1,h_2,h_3,h_4$ are exponential rates;
		\item[(iii)]\textit{the polynomial growth} ($p.g.$) and respectively \textit{uniform polynomial growth} ($u.p.g.$) when $h_1,h_2,h_3,h_4$ are polynomial rates;
	\end{itemize}
\end{remark}
A characterization of the H-g is given by
\begin{proposition}\label{prop crestere}
	The pair $(U,\mathcal{P})$ has H-g if and only if there exists a nondecreasing function $ M:\mathbb{R}_+\to[1,\infty)$ such that:
	\begin{itemize}
		\item [$(h_1g_2)$ ]$h_1(s)\|U(t,s)P_1(s)x\|\leq M(s)h_1(t)\|P_1(s)x\|$
		\item [$(h_2g_2)$ ]$h_2(s)\|P_2(s)x\|\leq M(t) h_2(t)  \|U(t,s)P_2(s)x\|$
		\item [$(h_3g_2)$ ]$h_3(s)\|U(t,s)P_3(s)x\|\leq M(s) h_3(t)\|P_3(s)x\|$
		\item[$(h_4g_2)$ ]$h_4(s)\|P_3(s)x\|\leq M(t) h_4(t)  \|U(t,s)P_3(s)x\|$
	\end{itemize}
for all $ (t,s,x)\in \Delta\times X$.
\end{proposition}
\begin{proof}
	\textit{Necessity:} The implications $(h_1g_1)\Rightarrow(h_1g_2)$ and $(h_3g_1)\Rightarrow(h_3g_2)$ result by replacing $x$ by $P_1(s)x$ in $(h_1g_1)$ respectively in $(h_3g_1)$.
	
	For $(h_2g_1)\Rightarrow(h_2g_2)$ and $(h_4g_1)\Rightarrow(h_4g_2)$ we observe that by Remark \ref{rem-proiectorstrong} and Definition \ref{def-crestere} it follows that
	\begin{align*}
	h_j(s)\|P_j(s)x\|&=h_j(s)\|V_j(t,s)U(t,s)P_j(s)x\|\leq M(t)h_j(t)\|U(t,s)P_j(s)x\|
	\end{align*}for all $(t,s,x)\in\Delta\times X$ and $j\in\{2,4\}$.

	\textit{Sufficiency:}	We denote by
	$$M_1(t)=\sup_{s\in[0,t]}M(s)(\|P_1(s)\|+\|P_2(s)\|+\|P_3(s)\|).$$
	
	For implications $(h_ig_2)\Rightarrow(h_ig_1)$, where $i\in\{1,3\}$ we observe that 	
	\begin{equation*}
		h_i(s)\|U(t,s)P_i(s)x\|\leq M(s)h_i(t)\|P_i(s)x\|\leq M_1(s)h_i(t)\|x\|
	\end{equation*} 
	for all $(t,s,x)\in\Delta\times X$ and $i\in\{1,3\}$.
	
	Similarly for $(h_jg_2)\Rightarrow(h_jg_1)$, where $j\in\{2,4\}$ we have
	\begin{align*}
	h_j(s)\|V_j(t,s)P_j(t)x\|&=h_j(s)\|P_j(s)V_j(t,s)P_j(t)x\|\leq M(t)h_j(t)\|U(t,s)P_j(s)V_j(t,s)P_j(t)x\|\\
	&\leq M_1(t)h_j(t)\|x\|
	\end{align*}for all $(t,s,x)\in\Delta\times X$. 
\end{proof}
\begin{corollary}\label{cor-crestere-unif}
	If the pair $(U,\mathcal{P})$ has an uniform-H-growth then there exists $M\geq 1$ such that
	\begin{itemize}
		\item [$(uHg_1)$ ]$h_i(s)\|U(t,s)P_i(s)x\|\leq Mh_i(t)\|P_i(s)x\|$
		\item [$(uHg_2)$ ]$h_j(s)\|P_j(t)x\|\leq M h_j(t)  \|U(t,s)P_j(s)x\|$
	\end{itemize}
	for all $ (t,s,x)\in \Delta\times X, i\in\{1,3\}$ where $j\in\{2,4\}$.
\end{corollary}
\begin{proof}
	It results from the proof of the previous proposition.
\end{proof}
\section{Trichotomy with growth rates}
With the notations from the previous section we introduce the main concept studied in this paper by
\begin{definition}\label{def-tricho}
	We say that the pair $(U,\mathcal{P})$ is  \textit{H-trichotomic} (and we denote H-t) if there exists a nondecreasing function $N:\mathbb{R}_+\to [1,\infty)$ such that:
\begin{itemize}
	\item [$(h_1t_1)$ ]$h_1(t)\|U(t,s)P_1(s)x\|\leq N(s)h_1(s)\|x\|$
	\item [$(h_2t_1)$ ]$h_2(t)\|V_2(t,s)P_2(t)x\|\leq N(t) h_2(s)  \|x\|$
	\item [$(h_3t_1)$ ]$h_3(s)\|U(t,s)P_3(s)x\|\leq N(s) h_3(t)\|x\|$
	\item[$(h_4t_1)$ ]$h_4(s)\|V_3(t,s)P_3(t)x\|\leq N(t) h_4(t)  \|x\|$
\end{itemize}
for all	$(t,s,x)\in \Delta\times X.$
\end{definition}
In particular if the function $N$ is constant then we obtain the \textit{uniform-H-trichotomy} concept, denoted by u-H-t.	
\begin{remark}
	As important particular cases of H-trichotomy we have:
	\begin{itemize}
		\item[(i)] \textit{(nonuniform) exponential trichotomy} ($e.t.$) and respectively \textit{uniform exponential trichotomy} ($u.e.t.$) when $h_1,h_2,h_3,h_4$ are exponential rates;
		\item[(ii)]\textit{(nonuniform) polynomial trichotomy} ($p.t$) and respectively \textit{uniform polynomial trichotomy} ($u.p.t.$) when $h_1,h_2,h_3,h_4$ are polynomial rates;
		\item[(iii)]\textit{(nonuniform) $(h_1,h_2)$-dichotomy} ($(h_1,h_2)-d$) and respectively \textit{uniform $(h_1,h_2)$-dichotomy} ($u-(h_1,h_2)-d$) for $P_3=0$;
		\item[(iv)] \textit{(nonuniform) exponential dichotomy} ($e.d.$) and respectively \textit{uniform exponential dichotomy} ($u.e.d.$) when $P_3=0$ and $h_1,h_2$ are exponential rates;
		\item[(v)]\textit{(nonuniform) polynomial dichotomy} ($p.d.$) and respectively \textit{uniform polynomial dichotomy} ($u.p.d$) for $P_3=0$ and $h_1,h_2$ are polynomial rates;
		\end{itemize}
\end{remark}
A characterization of H-trichotomy is given by
\begin{proposition}\label{prop tricho}
	The pair $(U,\mathcal{P})$ is H-trichotomic if and only if there exists a nondecreasing function $ N:\mathbb{R}_+\to[1,\infty)$ such that:
	\begin{itemize}
		\item [$(h_1t_2)$ ]$h_1(t)\|U(t,s)P_1(s)x\|\leq N(s)h_1(s)\|P_1(s)x\|$
		\item [$(h_2t_2)$ ]$h_2(t)\|P_2(s)x\|\leq N(t) h_2(s)  \|U(t,s)P_2(s)x\|$
		\item [$(h_3t_2)$ ]$h_3(s)\|U(t,s)P_3(s)x\|\leq N(s) h_3(t)\|P_3(s)x\|$
		\item[$(h_4t_2)$ ]$h_4(s)\|P_3(s)x\|\leq N(t) h_4(t)  \|U(t,s)P_3(s)x\|$
	\end{itemize}
	for all $ (t,s,x)\in \Delta\times X$.
\end{proposition}
\begin{proof}
It is similar with the proof of the Proposition \ref{prop crestere}.
\end{proof}
\begin{corollary}
	If the pair $(U,\mathcal{P})$ has an uniform-H-trichotomy then there exists $N\geq 1$ such that
	\begin{itemize}
		\item [$(uh_1t_2)$ ]$h_1(t)\|U(t,s)P_1(s)x\|\leq Nh_1(s)\|P_1(s)x\|$
		\item [$(uh_2t_2)$ ]$h_2(t)\|P_2(s)x\|\leq N h_2(s)  \|U(t,s)P_2(s)x\|$
		\item [$(uh_1t_2)$ ]$h_3(s)\|U(t,s)P_3(s)x\|\leq Nh_3(t)\|P_3(s)x\|$
		\item [$(uh_2t_2)$ ]$h_4(s)\|P_3(s)x\|\leq N h_4(t)  \|U(t,s)P_3(s)x\|$
	\end{itemize}
	for all $ (t,s,x)\in \Delta\times X.$
\end{corollary}
\begin{proof}
It is similar with the proof of Corollary \ref{cor-crestere-unif}.	
\end{proof}

\section{Examples and counterexamples}

In this section we consider the set $H=\{h_1,h_2,h_3,h_4\}$ and three orthogonal projectors $P_1,P_2,P_3\in\mathcal{B}(X)$ (i.e. $P_1+P_2+P_3=I, P_i^2=P_i$ for every $i\in\{1,2,3\}$ and $P_iP_j=0$ for $i\neq j$). The connections between the concepts of trichotomy and growths defined in the previous sections are given in the following diagram
\[
\begin{matrix}
	u.-H-t & \Rightarrow &H-t\\
	\Downarrow& & \Downarrow\\
		u.-H-g & \Rightarrow&H-g.
\end{matrix}
\]
 The aim of this section is to show that the converse implications are not valid.
 \begin{example}[\textit{A pair $(U,\mathcal{P})$ which has an uniform-H-growth and is not H-trichotomic}]
Let  $U:\Delta\to\mathcal{B}(X)$ be the evolution operator defined by
\begin{align}
U(t,s)=\frac{h_1^2(t)}{h_1^2(s)}P_1+ \frac{h_2(s)}{h_2(t)}P_2+\frac{h_3(t)}{h_3(s)}\frac{h_4(s)}{h_4(t)}P_3 
\end{align}for all $(t,s)\in\Delta$.

Then $\mathcal{P}=\{P_1, P_2, P_3\}$ is compatible with $U$ and the pair $(U,\mathcal{P})$
has the properties:
\begin{align*}
(uh_1g_1)\thickspace \thickspace h_1(s)\|U(t,s)P_1(s)x\|&=\frac{h_1^2(t)}{h_1(s)}\|P_1(s)x\|\leq h_1(t) \|P_1\|\|x\|\leq Nh_1(t)\|x\|\\
 (uh_2g_1)\thickspace \thickspace h_2(s)\|V_2(t,s)P_2(t)x\|&={h_2(t)}\|P_2(s)x\|\leq  h_2(t)\|P_2\|\|x\|\leq Nh_2(t)\|x\|\\
 (uh_3g_1)\thickspace \thickspace h_3(s)U(t,s)P_3(s)x\|&=\frac{h_3(t)h_4(s)}{h_4(t)}\|P_3(s)x\|\leq Nh_3(t)\|x\|\\
(uh_4g_1)\thickspace \thickspace h_4(s)V_3(t,s)P_3(t)x\|&=\frac{h_4(t)h_3(s)}{h_3(t)}\|P_3(s)x\|\leq N h_4(t)\|x\|
\end{align*} for all $(t,s,x)\in\Delta\times X$, where 
$$N=\|P_1\|+\|P_2\|+\|P_3\|.$$

In consequence the pair $(U,\mathcal{P})$ has a u-H-g and hence H-g	.

The pair is not H-t, because if we assume on the contrary then by $(h_1t_1)$ there exists a function $N:\mathbb{R}_+\to[1,\infty)$ such that
\begin{equation*}
\frac{h_2^2(t)}{h_2(s)}\|P_2x\|=h_2(t)\|V_2(t,s)P_2x\|\leq N(s)h_2(s)\|P_2x\|
\end{equation*}which implies
\begin{equation*}
h_2^2(t)\leq N(s)h_2^2(s), \text{ for all } (t,s)\in\Delta.
\end{equation*}
This inequality is not valid ({it is sufficient to consider $s=0$ and $t\to\infty$}). Hence, $(U,\mathcal{P})$ is not H-t and in consequence is not u-H-t.

In conclusion this example shows that in general the implications
$$u-H-g\Rightarrow u-H-t$$ and 
$$H-g\Rightarrow H-t$$ are not valid.
\end{example}
\begin{example}[\textit{A pair $(U,\mathcal{P})$ which is H-trichotomic and has not an uniform H-growth}]
Let  $U:\Delta\to\mathcal{B}(X)$ be the evolution operator defined by
\begin{align}
U(t,s)=\frac{h_1^2(s)}{h_1^2(t)}P_1+ \frac{h_2^2(s)}{h_2^2(t)}P_2+\frac{h_2^3(s)h_3(t)}{h_2^3(t)h_3(s)}\frac{h_4(s)}{h_4(t)}P_3 
\end{align}for all $(t,s)\in\Delta$.

We observe that if denote
$$N(t)=(\|P_1\|+\|P_2\|+\|P_3\|)h_2^3(t)$$
then
\begin{align*}
(h_1t_1)\thickspace \thickspace h_1(t)\|U(t,s)P_1(s)x\|&=\frac{h_1^2(s)}{h_1(t)}\|P_1(s)x\|\leq h_1(s) \|P_1\|\|x\|\leq N(s)h_1(s)\|x\|,\\
(h_2t_1)\thickspace \thickspace h_2(t)\|V_2(t,s)P_2(s)x\|&=\frac{h_2^3(t)}{h_2^2(s)}\|P_2(s)x\|\leq   N(t)h_2(s)\|x\|,\\
(h_3t_1)\thickspace \thickspace h_3(s)\|U(t,s)P_3(s)x\|&=h_3(t)\frac{h_2^3(s)h_4(s)}{h_2^3(t)h_4(t)}\|P_3(s)x\|\leq N(s)h_3(t)\|x\|,\\
(h_4t_1)\thickspace \thickspace h_4(s)\|V_3(t,s)P_3(s)x\|&=h_4(t)\frac{h_3(s)h_2^3(t)}{h_3(t)h_2^3(s)}\|P_3(s)x\|\leq N(t) h_4(t)\|x\|,
\end{align*} for all $(t,s,x)\in\Delta\times X.$

Finally it follows that the pair $(U,\mathcal{P})$ is H-trichotomic.

If we assume that $(U,\mathcal{P})$ has an uniform H-growth then by $(uh_2g_1)$ it results that there exits $M\geq 1$ such that	
\begin{equation*}
\frac{h_2^2(t)}{h_2(s)}\|P_2x\|=h_2(s)\|V_2(t,s)P_2x\|\leq Mh_2(t)\|x\|, \text{ for all } (t,s)\in\Delta.
\end{equation*}
Thus we obtain
\begin{equation*}
h_2^2(t)\|P_2\|\leq Mh_2^2(s), \text{ for all } (t,s)\in\Delta.
\end{equation*}
which is impossible for $s=0$ and $t\to\infty$.

We conclude that $(U,\mathcal{P})$ has not H-growth. 

Furthermore, it follows that the implications
\begin{align*}
H-t&\Rightarrow u-H-t\\
H-g&\Rightarrow u- H-g\\
H-t&\Rightarrow  u-H-g
\end{align*} are not valid.
	
\end{example}

\section{The main result}
In this section we give a characterization of H-trichotomy in terms of a certain type of Lyapunov functions. Firstly we introduce
\begin{definition}\label{def-lyapunov-function}
	A function $L:\mathbb{R}_+\times X\to\mathbb{R}_+$ is called \textit{a H-Lyapunov function} for the pair $(U,\mathcal{P})$  if there exists a nondecreasing map $T:\mathbb{R}_+\to[1,\infty)$ such that 
	\begin{itemize}
\item[$(L_0)$  ]$\|x\|\leq L(t,x)\leq T(t)\|x\|;$
\item[$(h_1L)$  ]$h_1(t)L(t,U(t,s)P_1(s)x)\leq T(s)h_1(s)L(s,x);$
\item[$(h_2L)$  ]$h_2(t)L(s,V_2(t,s)P_2(t)x)\leq T(t)h_2(s)L(t,x);$
\item[$(h_3L)$  ]$h_3(s)L(t,U(t,s)P_3(s)x)\leq T(s)h_3(t)L(s,x);$
\item[$(h_4L)$  ]$h_4(s)L(s,V_3(t,s)P_3(t)x)\leq T(t)h_4(s)L(t,x),$
	\end{itemize}
		for all $(t,x)\in\mathbb{R}_+\times X$.
\end{definition}
We begin with
\begin{lemma}
	If the pair $(U,\mathcal{P})$ has a H-growth then the function $L:\mathbb{R}_+\times X\to\mathbb{R}_+$ defined by
	\begin{align}\label{functia-Lyapunov}
	L(t,x)&=\sup_{\tau\geq t} \frac{h_1(t)}{h_1(\tau)}\|U(\tau,t)P_1(t)x\|
	+\sup_{r\leq t} \frac{h_2(r)}{h_2(t)}\|V_2(t,r)P_2(t)x\|\nonumber\\
	&+\sup_{\tau\geq t} \frac{h_3(t)}{h_3(\tau)}\|U(\tau,t)P_3(t)x\|
	+\sup_{r\leq t} \frac{h_4(r)}{h_4(t)}\|V_3(t,r)P_3(t)x\|
	\end{align} satisfies the condition $(L_0)$.
\end{lemma}
\begin{proof}
	If $(U,\mathcal{P})$ has a H-growth then by Definition \ref{def-crestere} it follows that there is a nondecreasing function $M:\mathbb{R}_+\to[1,\infty)$ such that
	\begin{equation*}
	L(t,x)\leq T(t)\|x\|, 
	\end{equation*}where $T(t)=4M(t)$, for all $(t,x)\in\mathbb{R}_+\times X.$
	On the other hand from the definition of $L$ it results that
	\begin{equation*}
	L(t,x)\geq \|P_1(t)x\|+\|P_2(t)x\|+\|P_3(t)x\|\geq \|P_1(t)x+P_2(t)x+P_3(t)x\|= \|x\|, 
	\end{equation*}for all $(t,x)\in\mathbb{R}_+\times X.$
\end{proof}
\begin{theorem}\label{teorem1}
	The pair $(U,\mathcal{P})$ is H-trichotomic if and only if there exists a H-Lyapunov function $L:\mathbb{R}_+\times X\to\mathbb{R}_+$ for $(U,\mathcal{P})$.
\end{theorem}
\begin{proof}
	\textit{Necessity:}If the pair $(U,\mathcal{P})$ is H-trichotomic then it has a H-growth and by the previous lemma we have the function $L:\mathbb{R}_+\times X\to\mathbb{R}_+$ given by (\ref{functia-Lyapunov}) that satisfies $(L_0)$.
	
	In what follows we shall denote $T(t)=2N(t)$, where the function N is given by Definition \ref{def-tricho}. 
\begin{itemize}
	\item[	$\boldsymbol{(h_1L)}$]  We observe that by the definition of $L$ and the condition $(h_1t_1)$ it results
	\begin{align*}
	&h_1(t)L(t,U(t,s)P_1(s)x)=h_1(t)\sup_{\tau\geq t} \frac{h_1(t)}{h_1(\tau)}\|U(\tau,t)P_1(t)U(t,s)P_1(s)x\|\\
	&=\sup_{\tau\geq t} \frac{h_1^2(t)}{h_1(\tau)}\|U(\tau,s)P_1(s)x\|\leq N(s)\frac{h_1(t)h_1(s)}{h_1(\tau)}\|x\|\leq N(s)h_1(s)\|x\|\\
	&\leq T(s)h_1(s)L(t,x).
	\end{align*}
	\item[	$\boldsymbol{ (h_2L)}$]  Similarly, the inequality $(h_2t_1)$ implies that
	\begin{align*}
	&h_2(t)L(s,V_2(t,s)P_2(t)x)=h_2(t)\sup_{r\leq s} \frac{h_2(r)}{h_2(s)}\|V_2(s,r)P_2(s)V_2(t,s)P_2(t)x\|\\
	&=h_2(t)\sup_{r\leq s} \frac{h_2(r)}{h_2(s)}\|V_2(t,r)P_2(t)x\|
	\leq N(t)h_2(s)\|x\|\\
	&\leq T(t)h_2(s)L(t,x),\\
	\end{align*}for all $(t,s,x)\in\Delta\times X.$
	\item[	$\boldsymbol{ (h_3L)}$] Firstly, we observe that from $(h_3t_1)$ we obtain
\begin{align}
&h_3(s)L(t,U(t,s)P_3(s)x)=h_3(s)(\sup_{\tau\geq t} \frac{h_3(t)}{h_3(\tau)}\|U(\tau,t)P_3(t)U(t,s)P_3(s)x\|\nonumber\\
&+\sup_{r\leq t} \frac{h_4(r)}{h_4(t)}\|V_3(t,r)P_3(t)U(t,s)P_3(s)x\|)\nonumber\\
&=h_3(s)(\sup_{\tau\geq t} \frac{h_3(t)}{h_3(\tau)}\|U(\tau,s)P_3(s)x\|+\sup_{r\leq t} \frac{h_4(r)}{h_4(t)}\|V_3(t,r)P_3(t)U(t,s)P_3(s)x\|)\nonumber\\
&\leq N(s)h_3(t)L(s,x)+h_3(s)\sup_{r\leq t} \frac{h_4(r)}{h_4(t)}\|V_3(t,r)P_3(t)U(t,s)P_3(s)x\|,\label{ec1}
\end{align}for all $(t,s,x)\in\Delta\times X.$
\begin{itemize}
	\item [\textit{Case I.}] If $s\leq r\leq t\leq \tau$ then by $(h_3t_1)$ we have that
	\begin{align}
	&h_3(s)\sup_{s\leq r\leq t} \frac{h_4(r)}{h_4(t)}\|V_3(t,r)U(t,s)P_3(s)x\|\nonumber\\&
	=h_3(s)	\sup_{s\leq r\leq t} \frac{h_4(r)}{h_4(t)}\|V_3(t,r)U(t,r)U(r,s)P_3(s)x\|\nonumber\\
	&\leq N(s)h_3(s)\sup_{s\leq r\leq t}h_3(s) \frac{h_4(r)}{h_4(t)}\frac{h_3(r)}{h_3(s)}\|x\|\nonumber\\
	&\leq N(s)h_3(t)L(s,x),\label{ec2}
	\end{align}for all $(t,s,x)\in\Delta\times X.$
	\item[\textit{Case II.}] Similarly, if $r\leq s\leq t\leq \tau$ then by $(h_3t_1)$ we obtain
	\begin{align}
	&h_3(s)	\sup_{r\leq t} \frac{h_4(r)}{h_4(t)}\|V_3(t,r)U(t,s)P_3(s)x\|\nonumber\\
	&=h_3(s)	\sup_{r\leq t} \frac{h_4(r)}{h_4(t)}\|V_3(s,r)V_3(t,s)P_3(t)U(t,s)P_3(s)x\|\nonumber	\\
	&=h_3(s)\sup_{r\leq t} \frac{h_4(r)}{h_4(t)}\|V_3(s,r)P_3(s)x\|\leq N(s) h_3(s)\sup_{r\leq t} \frac{h_4(r)}{h_4(t)}\frac{h_4(s)}{h_4(r)}\|x\|\nonumber\\
	&\leq N(s)h_3(t)L(s,x),\label{ec3}
	\end{align}for all $(t,s,x)\in\delta\times X.$
\end{itemize}
The inequalities (\ref{ec1}), (\ref{ec2}) and (\ref{ec3}) show that $(h_4L)$ is satisfied for $T(t)=2N(t)$.
\item[	$\boldsymbol{(h_4L)}$] Using Remark \ref{rem-proiectorstrong} and the inequality $(h_4t_1)$ it follows that
\begin{align}
& h_4(s)L(s,V_3(t,s)P_3(t)x)=h_4(s)L(s,P_3(s)V_3(t,s)P_3(t)x)\nonumber\\
&=h_4(s)(\sup_{\tau\geq s} \frac{h_3(s)}{h_3(\tau)}\|U(\tau,s)P_3(s)V_3(t,s)P_3(t)x\|\nonumber\\
&+\sup_{r\leq s} \frac{h_4(r)}{h_4(s)}\|V_3(s,r)P_3(s)V_3(t,s)P_3(t)x\|)\nonumber\\
&=h_4(s)\sup_{\tau\geq s} \frac{h_3(s)}{h_3(\tau)}\|U(\tau,s)V_3(t,s)P_3(t)x\|+h_4(s)\sup_{r\leq s} \frac{h_4(r)}{h_4(s)}\|V_3(t,r)P_3(t)x\|\nonumber\\
&\leq h_4(s)\sup_{\tau\geq s} \frac{h_3(s)}{h_3(\tau)}\|U(\tau,s)V_3(t,s)P_3(t)x\|+N(t)h_4(s)\sup_{r\leq s}\frac{h_4(r)}{h_4(s)}\frac{h_4(t)}{h_4(r)}\|x\|\nonumber\\
&\leq h_4(s)\sup_{\tau\geq s} \frac{h_3(s)}{h_3(\tau)}\|U(\tau,s)V_3(t,s)P_3(t)x\|+N(t)h_4(t)L(t,x),\label{ec4}
\end{align}for all $(t,s,x)\in\Delta\times X.$
\begin{itemize}
	\item[\textit{Case I.}] If $\tau\geq t\geq s$ then by $(e_2)$, Remark \ref{rem-proiectorstrong} and $(h_3t_4)$ we obtain that
	\begin{align}
&	h_4(s)\sup_{\tau\geq s}\frac{h_3(s)}{h_3(\tau)}\|U(\tau,s)V_3(t,s)P_3(t)x\|\nonumber\\&= h_4(s)\sup_{\tau\geq s} \frac{h_3(s)}{h_3(\tau)}\|U(\tau,t)P_3(t)x\|\nonumber\\
	&\leq N(t)h_4(s)\sup_{r\geq s} \frac{h_3(s)}{h_3(\tau)}\frac{h_3(\tau)}{h_3(t)}\|x\|\leq N(t)h_4(t)\|x\|\nonumber\\
	&\leq N(t)h_4(t)L(t,x),\label{ec5}
	\end{align}for all $(t,s,x)\in\Delta\times X.$
	\item[\textit{Case II.}] Similarly, if $t\geq\tau\geq s$ then by Remark \ref{rem-proiectorstrong} and $(h_4t_1)$ it result that
	\begin{align}
	&h_4(s)\sup_{t\geq\tau\geq s}\frac{h_3(s)}{h_3(\tau)}\|U(\tau,s)V_3(t,s)P_3(t)x\|\nonumber\\
	&=h_4(s)\sup_{t\geq\tau\geq s}\frac{h_3(s)}{h_3(\tau)}\|U(\tau,s)V_3(\tau,s)P_3(\tau)V_3(t,\tau)P_3(t)x\|\nonumber\\
	&=h_4(s)\sup_{t\geq\tau\geq s}\frac{h_3(s)}{h_3(\tau)}\|V_3(t,\tau)P_3(t)x\|\nonumber\\
	&\leq N(t) h_4(s)\sup_{t\geq\tau\geq s}\frac{h_3(s)}{h_3(\tau)}\frac{h_4(t)}{h_4(\tau)}\|x\|\nonumber\\
	&\leq N(t)h_4(t)L(t,x),\label{ec6}
	\end{align}for all $(t,s,x)\in\Delta\times X.$
\end{itemize}
The inequalities (\ref{ec4}), (\ref{ec5}) and (\ref{ec6}) show that $(h_4L)$ is satisfied.
\end{itemize}	
\textit{Sufficiency:} Using Definition \ref{def-lyapunov-function} we obtain that
\begin{align*}
 h_1(t)\|U(t,s)P_1(s)x\|&\leq h_1(t)L(t,U(t,s)P_1(s)x)\\
&\leq T(s)h_1(s)L(s,x)\leq T^2(s)h_1(s)\|x\|;\\
 h_2(t)\|V_2(t,s)P_2(t)x\|&\leq h_2(t)L(s,V_2(t,s)P_2(t)x)\\&\leq T(t)h_2(s)L(t,x)\leq T^2(t)h_2(s)\|x\|;\\
 h_3(s)\|U(t,s)P_3(s)x\|&\leq h_3(s)L(t,U(t,s)P_3(s)x)\\&\leq T(s)h_3(t)L(s,t)\leq T^2(s)h_3(t)\|x\|;\\
 h_4(s)\|V_3(t,s)P_4(t)x\|&\leq h_4(s)L(t,V_3(t,s)P_4(t)x)\\&\leq T(t)h_4(t)L(t,x)\leq T^2(t)h_4(t)\|x\|,
\end{align*}for all $(t,s,x)\in\Delta\times X.$

Finally, by Definition \ref{def-tricho} it results that $(U,\mathcal{P})$ is H-trichotomic.
\end{proof}
As a particular case, we obtain
\begin{corollary}
	The pair $(U,\mathcal{P})$ is uniformly-H-trichotomic if and only if there exists an uniform H-Lyapunov function for $(U,\mathcal{P})$.
\end{corollary}
\begin{proof}
	It results from the proof of Theorem \ref{teorem1} for the particular case when the functions $M,N$ and $T$ are constant.
\end{proof}

\end{document}